\documentclass[a4paper]{amsart}

\usepackage{amsthm}
\usepackage{amsmath}
\usepackage{amssymb}
\usepackage{latexsym}
\usepackage{enumerate}
\usepackage{graphicx}
\usepackage[utf8]{inputenc}
\usepackage{epsfig} 
\usepackage{pgf}
\usepackage{tikz}
\usepackage{float}
\usepackage{dsfont}
\usepackage{times}

\newtheorem{theoremA}{Theorem}

\newtheorem{theoremName}{Theorem A}
\renewcommand{\thetheoremName}

\newtheorem{proposition[[]]}[theoremName]{Proposition G}

\newtheorem{theorem}{Theorem}[section]

\newtheorem{corollary}[theorem]{Corollary}

\newtheorem{examples}[theoremName]{Examples}

\theoremstyle{definition}
\newtheorem{definition}[theorem]{Definition}

\newtheorem{remark}{Remark}

\numberwithin{equation}{section}

\newcommand{\Dis}{\displaystyle}

\newcommand{\dist}{\operatorname{dist}}
\newcommand{\Vol}{\operatorname{Vol}}

\newcommand{\Sup}{\operatorname{Sup}}

\newcommand{\kan}{\mathbb{K}^{n}(b)}

\newcommand{\kam}{\mathbb{K}^{m}(b)}
\newcommand{\erre}{\mathbb{R}}

\newcommand{\gr}{\operatorname{\nabla}}
\begin{document}

\title[Cheeger Isoperimetric Constant]
{Volume growth of submanifolds  and the Cheeger Isoperimetric Constant}
\author[V. Gimeno]{Vicent Gimeno$^{\#}$}
\address{Departament de Matem\`{a}tiques-INIT, Universitat Jaume I, Castell\'o,
Spain.}
\email{gimenov@guest.uji.es}

\author[V. Palmer]{Vicente Palmer$^{\#}$}
\address{Departament de Matem\`{a}tiques-INIT, Universitat Jaume I, Castell\'o,
Spain.} \email{palmer@mat.uji.es}
\thanks{$^{\#}$ Supported by 
Fundaci\'o Caixa Castell\'o-Bancaixa Grants P1.1B2006-34 and P1.1B2009-14 and by MICINN grant No. MTM2010-21206-C02-02.}




\begin{abstract}
We obtain an estimate of the Cheeger isoperimetric constant in terms of the volume growth for a properly immersed submanifold in a Riemannian manifold which possesses at least one pole and sectional curvature bounded from above .
\end{abstract}

\maketitle

\section{Introduction}\label{secIntro}
The Cheeger isoperimetric constant $\mathcal{I}_\infty(M)$ (see \cite{Chee}) of a non-compact Riemannian manifold of dimension $n \geq 2$ is defined as:
\begin{equation}\label{CheegerConstant}
\mathcal{I}_\infty(M):=\inf\limits_\Omega\bigg\{  \frac{\Vol(\partial \Omega)}{\Vol(\Omega)} \bigg\}
\end{equation}
where $\Omega$ ranges over open submanifolds of $M$ possessing compact closure and smooth boundary, $\Vol(\partial \Omega)$ denotes the $(n-1)$-dimensional volume of the boundary $\partial \Omega$, and $\Vol(\Omega)$ denotes the $n$-dimensional volume of $\Omega$, (concerning this definition, see also \cite{Cha2} and \cite{Cha3}).

This paper focuses on obtaining sharp upper and lower bounds for the Cheeger isoperimetric constant $\mathcal{I}_\infty(P)$ of a complete submanifold $P$ with controlled mean curvature and properly immersed in an ambient manifold $N$ with sectional curvatures bounded from above and which possess at least one pole. 

As a consequence of these upper and lower bounds, and as a preliminary view of our main theorems (Theorems \ref{Main1} and \ref{Main2} in section \S.3), we present the following results, which constitute a particular case of them  when a complete, non-compact and minimal submanifold properly immersed in a Cartan-Hadamard manifold is considered. In contrast, if we focus on compact and minimal submanifolds of a Riemannian manifold satisfying other geometric restrictions, we refer to the work \cite{HoS}, where certain isoperimetric inequalities involving these submanifolds have been proven.

\begin{theoremA}\label{Glimpse1}
Let $P^m$ be a complete non-compact and minimal submanifold properly immersed in a Cartan-Hadamard manifold  $N$ with sectional curvatures bounded from above as $K_N \leq b \leq 0$, and suppose that  
$\Sup_{t>0}(\frac{\Vol(P\cap B^N_t)}{\Vol(B^{m,b}_t)})<\infty$, where $B^N_t$ is the geodesic $t$-ball in the ambient manifold $N$ and $B^{m,b}_t$ denotes the geodesic $t$-ball in the real space form of constant sectional curvature $\kam$.

Then 
\begin{equation}\label{statcorG1}
\mathcal{I}_{\infty}(P) \leq  (m-1)\sqrt{-b} \quad.
\end{equation}

\end{theoremA}

\begin{theoremA}\label{Glimpse2}
Let $P^m$ be a complete non-compact and minimal submanifold properly immersed in a Cartan-Hadamard manifold  $N$ with sectional curvatures bounded from above as $K_N \leq b \leq 0$.
Then 
\begin{equation}\label{statcorG1}
\mathcal{I}_{\infty}(P) \geq  (m-1)\sqrt{-b} \quad.
\end{equation}

\end{theoremA}

The lower bounds for $\mathcal{I}_\infty(P)$ in Theorem \ref{Glimpse2} come from direct application of the divergence theorem to the Laplacian of the extrinsic distance defined on the submanifold using the distance in the ambient manifold, following the arguments of Proposition 3 in \cite{Y} and of Theorem 6.4 in \cite{Cha3}.

On the other hand, the upper bounds in Theorem \ref{Glimpse1} were obtained by assuming that the (extrinsic) volume growth of the submanifold is bounded from above by a finite quantity. As we shall see in the corollaries, when the submanifold is a minimal immersion in the Euclidean space or when we are dealing with minimal surfaces in the Euclidean or the Hyperbolic space, this crucial fact relates Cheeger's constant $\mathcal{I}_\infty(P)$ with the total extrinsic curvature of the submanifold  $\int_P \Vert B^P\Vert^{m} d\sigma$, in the sense that the finiteness of this total extrinsic curvature implies the upper bounds for Cheeger's constant, using the results in \cite{A1}, \cite{QCh1} and \cite{QCh3}.

These lower and upper bounds of $\mathcal{I}_\infty(P)$ given in Theorems \ref{Main1} and \ref{Main2} 
come from comparisons for the Laplacian of the extrinsic distance defined on the submanifold, and the techniques used to obtain these comparisons  are based on the Hessian analysis of this restricted distance function. When the extrinsic curvature of the submanifold is bounded (from above or from below), this analysis focuses on the relation, given in \cite{GreW}, between the Hessian of this function and these (extrinsic) curvature bounds, thus providing comparison results for the Hessian and the Laplacian of the distance function in the submanifold. 

The model used in these comparisons is constructed from the corresponding values for these operators computed for the intrinsic distance of a rotationally symmetric space whose  sectional curvatures bound the corresponding curvatures of the ambient manifold. 

We shall see that the Cheeger constant $\mathcal{I}_\infty(P)$ is bounded by the limit of some isoperimetric quotient determined by the geodesic $r$-balls in these model spaces, which involves the mean curvature of the submanifold.
\subsection{Outline of the paper}
 In section \S.2 we present the basic definitions and facts concerning the extrinsic distance restricted to a submanifold, and about the rotationally symmetric spaces used as a model for comparison. We also present the basic results regarding the Hessian comparison theory of restricted distance function that will be used. This section finishes with the description of the isoperimetric context where the results hold. Section \S.3 is devoted to the statement and proof of the two main Theorems \ref{Main1} and \ref{Main2} and three corollaries are stated and proven in the final section \S.4. 
 
\bigskip
\section{Preliminaires} \label{Prelim}  

\subsection{The extrinsic distance}

We assume throughout the paper that $P^{m}$ is a complete,  non-compact, properly immersed,
$m$-dimensional submanifold in a complete Riemannian manifold $N^n$ which possesses at least one pole $o\in N$. Recall that a pole
is a point $o$ such that the exponential map
$$\exp_{o}\colon T_{o}N^{n} \to N^{n}$$ is a
diffeomorphism. For every $x \in N^{n}\setminus \{o\}$ we
define $r(x) = r_o(x) = \dist_{N}(o, x)$, and this
distance is realized by the length of a unique
geodesic from $o$ to $x$, which is the {\it
radial geodesic from $o$}. We also denote by $r$
the restriction $r\vert_P: P\to \erre_{+} \cup
\{0\}$. This restriction is called the
{\em{extrinsic distance function}} from $o$ in
$P^m$. The gradients of $r$ in $N$ and $P$ are
denoted by $\gr^N r$ and $\gr^P r$,
respectively. Let us remark that $\gr^P r(x)$
is just the tangential component in $P$ of
$\gr^N r(x)$, for all $x\in S$. Then we have
the following basic relation:
\begin{equation}\label{radiality}
\nabla^N r = \gr^P r +(\gr^N r)^\bot 
\end{equation}
where $(\gr^N r)^\bot(x)=\nabla^\bot r(x)$ is perpendicular to
$T_{x}P$ for all $x\in P$.

\begin{definition}\label{ExtBall}
Given a connected and complete
submanifold $P^m$ properly immersed in a manifold $N^n$ with a pole $o \in N$, we
denote
the {\em{extrinsic metric balls}} of radius $t >0$ and center $o \in N$ by
$D_t(o)$. They are defined as  the intersection
$$
B^N_{t}(o) \cap P =\{x\in P \colon r(x)< t\}
$$
where $B^N_{t}(o)$ denotes the open geodesic ball
of radius $t$ centered at the pole $o$ in
$N^{n}$.
\end{definition}

\begin{remark}\label{theRemk0}
The extrinsic domains $D_t(o)$
are precompact sets (because in the definition above it was assumed that  the submanifold $P$ is properly immersed), 
with smooth boundary $\partial D_t(o)$.  The assumption on the smoothness of
$\partial D_{t}(o)$ makes no restriction. Indeed, 
the distance function $r$ is smooth in $N \setminus \{o\}$ 
since $N$ is assumed to possess a pole $o\in N$. Hence
the restriction $r\vert_P$ is smooth in $P$ and consequently the
radii $t$ that produce smooth boundaries
$\partial D_{t}(o)$ are dense in $\mathbb{R}$ by
Sard's theorem and the Regular Level Set Theorem.

\end{remark}

We now present the curvature restrictions which constitute the geometric framework of our study.

\begin{definition}
Let $o$ be a point in a Riemannian manifold $N$
and let $x \in N-\{ o \}$. The sectional
curvature $K_{N}(\sigma_{x})$ of the two-plane
$\sigma_{x} \in T_{x}N$ is then called a
\textit{$o$-radial sectional curvature} of $N$ at
$x$ if $\sigma_{x}$ contains the tangent vector
to a minimal geodesic from $o$ to $x$. We denote
these curvatures by $K_{o, N}(\sigma_{x})$.
\end{definition}

In order to control the mean curvatures $H_P(x)$ of $P^{m}$ at distance $r$ from
$o$ in $N^{n}$ we introduce the following definition:

\begin{definition} The $o$-radial mean curvature function for $P$ in $N$
is defined in terms of the inner product of $H_{P}$ with the $N$-gradient of the
distance function $r(x)$ as follows:
$$
\mathcal{C}(x) = -\langle \nabla^{N}r(x),
H_{P}(x) \rangle  \quad {\textrm{for all}}\quad x
\in P \,\, .
$$
\end{definition} 

\subsection{Model Spaces} \label{secModel}

The model spaces $M^m_w$ are rotationally symmetric spaces which serve
as com\-pa\-ri\-son controllers for the radial sectional
curvatures of the ambient space $N^{n}$.

\begin{definition}[see \cite{Gri}, \cite{GreW}]
 A $w-$model $M_{w}^{m}$ is a
smooth warped product with base $B^{1} = [\,0,\, R[ \,\,\subset\,
\mathbb{R}$ (where $\, 0 < R \leq \infty$\,), fiber $F^{m-1} =
S^{m-1}_{1}$ (i.e., the unit $(m-1)-$sphere with standard metric),
and warping function $w:\, [\,0, \,R[\, \to \mathbb{R}_{+}\cup
\{0\}\,$ with $w(0) = 0$, $w'(0) = 1$, and $w(r) > 0\,$ for all
$\, r > 0\,$. The point $o_{w} = \pi^{-1}(0)$, where $\pi$ denotes
the projection onto $B^1$, is called the {\em{center point}} of
the model space. If $R = \infty$, then $o_{w}$ is a pole of
$M_{w}^{m}$.
\end{definition}

\begin{remark}\label{propSpaceForm}
The simply connected space forms $\mathbb{K}^{m}(b)$ of constant
curvature $b$ can be constructed as  $w-$models $\kan =M^n_{w_b}$ with any given
point as the center point using the warping functions
\begin{equation}
w_b(r) =\begin{cases} \frac{1}{\sqrt{b}}\sin(\sqrt{b}\, r) &\text{if $b>0$}\\
\phantom{\frac{1}{\sqrt{b}}} r &\text{if $b=0$}\\
\frac{1}{\sqrt{-b}}\sinh(\sqrt{-b}\,r) &\text{if $b<0$} \quad .
\end{cases}
\end{equation}
Note that for $b > 0$ the function $w_{b}(r)$ admits a smooth
extension to  $r = \pi/\sqrt{b}$. For $\, b \leq 0\,$ any center
point is a pole.
\end{remark}
\begin{remark}
The sectional curvatures of the model spaces $K_{o_{w} , M_{w}}$ in the radial directions from the center
point are determined
by the radial function ${\Dis K_{o_{w} , M_{w}}(\sigma_{x}) \, = \, K_{w}(r) \, = \,
-\frac{w''(r)}{w(r)}}$, (see  \cite{GreW}, \cite{Gri} \cite{O'N}). Moreover,
 the mean curvature of the distance sphere of radius $r$ from the center point is
\begin{equation}\label{eqWarpMean}
\eta_{w}(r)  = \frac{w'(r)}{w(r)} = \frac{d}{dr}\ln(w(r))\quad .
\end{equation} \\

Hence, the sectional curvature of $\kan$ is given by ${\Dis -\frac{w_b''(r)}{w_b(r)}=b}$ and  the mean curvature of the geodesic $r-$sphere $S^{w_b}_r=S^{b,n-1}_r$ in the real space form $\kan$, \lq pointed inward' is (see \cite{Pa}):
$$
\eta_{w_b}=h_b(t)=\left\{
\begin{array}{l}
\sqrt{b}\cot\sqrt{b}t\,\,\text{  if }\,\,b>0\\
1/t\,\,\text{  if }\,\, b=0\\
\sqrt{-b}\coth\sqrt{-b}t\,\, \text{  if }\,\, b<0 \quad.
\end{array}\right.
$$
\end{remark}

In particular, in \cite{MP2} we introduced,
for any given warping function $\,w(r)\,$, the
isoperimetric quotient function $\,q_{w}(r)\,$  for the
corresponding $w-$mo\-del space $\,M_{w}^{m}\,$ as follows:
\begin{equation} \label{eqDefq}
q_{w}(r) \, = \, \frac{\Vol(B_{r}^{w})}{\Vol(S_{r}^{w})} \, = \,
\frac{\int_{0}^{r}\,w^{m-1}(t)\,dt}{w^{m-1}(r)} \quad .\\
\end{equation}
\noindent where $B^{w}_r$ and $S^{w}_r$ denotes  the metric $r-$ball and the metric $r-$sphere in $M^m_w$ respectively.
\subsection{Hessian comparison analysis of the extrinsic distance}\label{subsecLap}

 This subsection offers a corollary
of the Hessian comparison Theorem A in \cite{GreW}, which concerns the bounds for the Laplacian of a radial function defined on the submanifold (see \cite{HMP} and  \cite{Pa2} for detailed computations, see also \cite{JK}).
\begin{theorem} \label{corLapComp} Let $N^{n}$ be a manifold with a pole $o$ and let $M_{w}^{m}$ denote a $w-$model
with center $o_{w}$. Let $P^m$ be a properly immersed submanifold in $N$. Then we have the following dual Laplacian inequalities for modified distance functions $f\circ r: P \longrightarrow \erre$:\\

 Suppose that every $o$-radial sectional curvature at $x
\in N - \{o\}$ is bounded  by the $o_{w}$-radial sectional
curvatures in $M_{w}^{m}$ as follows:
\begin{equation}
\mathcal{K}(\sigma(x)) \, = \, K_{o, N}(\sigma_{x})
\leq-\frac{w''(r)}{w(r)}\quad .
\end{equation}

Then we have for every smooth function $f(r)$ with $f'(r) \leq
0\,\,\textrm{for all}\,\,\, r$, (respectively $f'(r) \geq
0\,\,\textrm{for all}\,\,\, r$):
\begin{equation} \label{eqLap2}
\begin{aligned}
\Delta^{P}(f \circ r) \, \leq (\geq) \, &\left(\, f''(r) -
f'(r)\eta_{w}(r) \, \right)
 \Vert \nabla^{P} r \Vert^{2} \\ &+ mf'(r) \left(\, \eta_{w}(r) +
\langle \, \nabla^{N}r, \, H_{P}  \, \rangle  \, \right)  \quad ,
\end{aligned}
\end{equation}
where $H_{P}$ denotes the mean curvature vector of $P$ in $N$.
\end{theorem}

\subsection{The Isoperimetric Comparison space}\label{secIsopCompSpace} 
We are going to define a new kind of model spaces, $M^m_W$. The limit  ${\Dis \lim_{r \to \infty} \frac{W'(r)}{W(r)}}$ of the quotient determined by its warping function (this quotient is given in terms of the mean curvature of the geodesic spheres in $M^m_W$ and the bounds on the mean curvature of the submanifold ${\Dis P}$) will serve as estimate for the isoperimetric constant $\mathcal{I}_\infty(P)$.

\begin{definition}[ \cite{MP3}] \label{defCspace}
Given the smooth functions $w: \erre_+ \longrightarrow \erre_+$ and $h:\erre_+ \longrightarrow \erre$ with $w(0)=0$, $w'(0)=1$ and $-\infty < h(0) < \infty$, the {\em{isoperimetric comparison space}} $M^m_W$ is
the $W-$model space with base interval $B\,= \, [\,0, R\,]$ and
warping function $W(r)$ defined by
the following differential
equation:
\begin{equation} \label{eqLambdaDiffeq}
\frac{W'(r)}{W(r)}\, = \,\eta_w(r)\,-\, \frac{m}{m-1} h(r) \\
\quad.
\end{equation}
and the following boundary condition:
\begin{equation} \label{eqTR}
W'(0) = 1
\quad .
\end{equation}

\end {definition}
By using equation (\ref{eqTR}), it is straightforward to see that $W(r) = 0$ only at $r=0$, so $M^m_W$ has a well-defined pole $o_{W}$ at $r = 0$. Moreover, $W(r) > 0$ for all $r>0$.\\


Note that when $h(r)=0\,{\textrm{for all }} r$,
then $W(r)=w(r) \,{\textrm{for all }} r$, so $M^m_W$ becomes
a model space with warping function $w$, $M^m_w$.\\

\begin{definition} \label{defBalCond1}
The model space $M_{W}^{m} $ is
{\em{$w-$balanced from above}} (with respect to the intermediary model space $M_{w}^{m}$)  iff
the following holds for all $r \in \, [\,0, R\,]$:
\begin{equation}\label{eqBalB}
\begin{aligned}
\eta_{w}(r) &\geq 0\\
\eta'_W(r) &\leq 0\,\,\forall r \, \quad .
\end{aligned}
\end{equation}
Note that $\eta'_W(r) \leq 0\,\,\forall r $ is equivalent to the condition 
\begin{equation}\label{eqBalB1}
-(m-1)(\eta^2_{w}(r) + K_w(r) ) \, \leq m h'(r) \quad .
\end{equation}
\end{definition}

\begin{definition} \label{defBalCond2}
The model space $M_{W}^{m} \,$ is
{\em{$w-$balanced from below}} (with respect to the intermediary model space $M_{w}^{m}$) iff
the following holds for all $r \in \, [\,0, R\,]$:
\begin{equation}\label{eqBalA}
q_{W}(r)\left(\eta_{w}(r) - h(r) \right) \, \geq 1/m \quad .\\
\end{equation}
\end{definition}

\begin{examples}
The following is a list of examples of isoperimetric comparison spaces and balance.
\begin{enumerate}

\item Given the functions $w_b(r)$ and $h(r)=C \geq \sqrt{-b}, \,\,\,\forall r>0$, let us consider $\kam= M^m_{w_b}$ as an intermediary model space with constant sectional curvature $b <0$. Then, it is straightforward to check that the model space $M^m_{W}$ defined from $w_b$ and $h$ as in Definition \ref{defCspace} is $w_b-$balanced from above, and is not $w_b-$balanced from below.
\medskip

\item Let $M^m_w$ be a model space, with $w(r)=e^{r^2}+r-1$. Let us now consider  $h(r)=0 \,\,\forall r>0$. In this case, as $h(r)=0$, then $W(r)=w(r)$, so the isoperimetric comparison space $M^m_W$ agrees with its corresponding intermediary model space $M^m_w$. Moreover, (see \cite{MP2}), 
$$q_w(r)\eta_w(r) \geq \frac{1}{m}\quad .$$ 
so $M^m_w$ is $w$-balanced from below. 

However, it is easy to see
that $\eta_w(r)=\frac{2re^{r^2}+1}{e^{r^2}+r-1}$ is an increasing function from a given value $r_0 >0$ and, hence, does not satisfy second inequality in (\ref{eqBalB}) and is therefore not $w$-balanced from above.
\medskip

\item Let $\kam=M^m_{w_b}$, ($b \leq 0$), be the Euclidean or Hyperbolic space, with warping function $w_b(r)$. Let us consider $h(r)=0 \,\,\,\forall r$. In this context, these spaces are isoperimetric spaces with themselves as intermediary spaces, and satisfy both balance conditions given in definitions \ref{defBalCond1} and \ref{defBalCond2} (see \cite{MP2}).\\

\end{enumerate}
\end{examples}

\subsection{Comparison Constellations}

We now present the precise settings
where our main results take place, and introduce the notion of {\em comparison constellations}.

\begin{definition}\label{defConstellatNew2}
Let $N^{n}$ denote a Riemannian manifold with a
pole $o$ and distance function $r \, = \, r(x) \,
= \, \dist_{N}(o, x)$. Let $P^{m}$ denote a complete and properly  immersed submanifold in
$N^{n}$. Suppose the following  conditions are
satisfied for all $x \in P^{m}$ with $r(x) \in
[\,0, R]\,$:
\begin{enumerate}[(a)]
\item The $o$-radial sectional curvatures of $N$ are bounded from above
by the $o_{w}$-radial sectional curvatures of
the $w-$model space $M_{w}^{m}$:
$$
\mathcal{K}(\sigma_{x}) \, \leq \,
-\frac{w''(r(x))}{w(r(x))} \quad .
$$

\item The $o$-radial mean curvature of $P$ is bounded from above by
a smooth radial function, (the {\em bounding function}) $h: \erre_+ \longrightarrow  \erre$, ($h(0) \in ]-\infty,\infty[$):
$$
\mathcal{C}(x)  \leq h(r(x)) \quad.
$$
\end{enumerate}

Let $M_W^{m}$ denote the $W$-model with the
specific warping function $W: \pi(M^m_W)
\to \mathbb{R}_{+}$ constructed in
Definition \ref{defCspace} via $w$, and $h$.
Then the triple $\{ N^{n}, P^{m}, M^m_W
\}$ is called an {\em{isoperimetric comparison
constellation}} on the interval $[\,0, R]\,$.
\end{definition}

\begin{examples} 
Minimal and non-minimal settings will now be described.

\begin{enumerate}

\item Minimal submanifolds immersed in an ambient Cartan-Hadamard manifold: let $P$ be a minimal submanifold of a Cartan-Hadamard manifold $N$, with sectional curvatures bounded above by $b \leq 0$. Let us consider the function  $h(r)=0\,\,\forall r \geq 0$ as the bounding function for the $o$-radial mean curvature of $P$ and the functions $w_b(r)$ with $b \leq 0$ as the warping function $w(r)$.

It is straigthforward to see that, under these restrictions, $W=w_b$ and, hence, $M^m_{W}=\kam$, so $\,\{N^{n},\,P^{m},\,\kam \} \,$ is an {\em{isoperimetric comparison
constellation}} on the interval $[\,0, R]\,$, for all $R > 0$. Here the model space $M^m_{W}=M^m_{w_b}=\kam$ is $w_b$-balanced from above and from below.

\medskip

\item Non-minimal submanifolds immersed in an ambient Cartan-Hadamard manifold. Let us consider again a Cartan-Hadamard manifold $N$, with sectional curvatures bounded above by $a \leq 0$. Let $P^m$ be a properly immersed submanifold in $N$ such that $$
\mathcal{C}(x)  \leq h_{a,b}(r(x)) \quad.
$$
where, by fixing $a<b<0$,  we define $h_{a,b}(r)=\frac{m-1}{m}(\eta_{w_a}(r)-\eta_{w_b}(r)) \,\,\forall r>0$. 

Then, it is straightforward to check that $W=w_b$ and, hence, $M^m_{W}=\kam$, so $\{ N^n,P^m, M^m_W\}$ is an {\em{isoperimetric comparison
constellation}} on the interval $[\,0, R]\,$, for all $R > 0$. Moreover the model space $M^m_{W}=M^m_{w_b}=\kam$ is $w_a$-balanced from above and from below.
\end{enumerate}

\end{examples}
\bigskip

\section{Main results}\label{Main}

Before stating our main theorems, we find the upper bounds for the isoperimetric quotient defined as the
volume of the extrinsic sphere divided by the volume of the
extrinsic ball, in the setting given by the comparison
constellations. 

\begin{theorem} \label{thmIsopGeneral1} (see \cite{HMP}, \cite{Pa}, \cite{MP})
 Consider an isoperimetric comparison
constellation $\{ N^{n}, P^{m}, M^m_W
\}$.
 Assume that the isoperimetric comparison space
 $\,M^m_W\,$ is $w$-balanced from below.
Then
\begin{equation} \label{eqIsopGeneralB}
\frac{\Vol(\partial D_{t})}{\Vol(D_{t})} \geq
\frac{\Vol(
S^{W}_{t})}{\Vol(B^{W}_{t})} \quad.
\end{equation}

Furthermore, the function $f(t)=\frac{\Vol(D_{t})}{\Vol(B^{W}_{t})}$ is monotone non-decreasing in $t$.

Moreover, if equality holds in (\ref{eqIsopGeneralB}) for some fixed radius $t_0 >0$, then $D_{t_0}$ is a cone
in the ambient space $N^n$.
\end{theorem}

The following is the upper bound for the Cheeger constant of a submanifold $P$:
\begin{theorem}\label{Main1} 
 Consider an isoperimetric comparison
constellation $\{ N^{n}, P^{m},M^m_W
\}$.
 Assume that the isoperimetric comparison space
 $\, M^m_W\,$ is $w$-balanced from below.
Assume, moreover, that 

\begin{enumerate}
\item  $\Sup_{t>0}(\frac{\Vol(D_t)}{\Vol(B^{W}_t)})<\infty$.
\item The limit $\lim_{t \to \infty}\frac{\Vol(S^{W}_t)}{\Vol(B^{W}_t)} $ exists
\end{enumerate}

Then 

\begin{equation}\label{ineq1}
\mathcal{I}_{\infty}(P) \leq \lim_{r \to \infty}\frac{\Vol(S^{W}_t)}{\Vol(B^{W}_t)} \quad.
\end{equation}

In particular, let $P^m$ be a complete and minimal submanifold properly immersed in a Cartan-Hadamard manifold  $N$ with sectional curvatures bounded from above as $K_N \leq b \leq 0$, and suppose that  $\Sup_{t>0}(\frac{\Vol(D_t)}{\Vol(B^{m,b}_t)})<\infty$.

Then 
\begin{equation}\label{statcor1}
\mathcal{I}_{\infty}(P) \leq (m-1)\sqrt{-b} \quad.
\end{equation}

\end{theorem}
\begin{proof}

Let us define 
\begin{equation}\label{def}
F(t):=\frac{\Vol(D_t)'}{\Vol(D_t)}-\frac{\Vol(S_t^{W})}{\Vol(B_t^{W})}=\left[\ln\left(\frac{\Vol(D_t)}{\Vol(B^W_t)}\right)\right]'
\end{equation}
By the co-area formula and applying Theorem \ref{thmIsopGeneral1} it is easy to see  that $F(t)$ is a non-negative function. Moreover, $\frac{\Vol(D_t)}{\Vol(B^{W}_t)}$ is non-decreasing (see \cite{MP}).

Integrating between $t_0 >0$ and $t>t_0$:
$$
\frac{\Vol(D_t)}{\Vol(B^W_t)}=\frac{\Vol(D_{t_0})}{\Vol(B^W_{t_0})}\, e^{\int_{t_0}^t F(s)\, ds}
$$
But on the other hand, from hypothesis (2) and the fact that  $\frac{\Vol(D_t)}{\Vol(B^{W}_t)}$ is non-decreasing, we know that $\lim_{t\to\infty}\frac{\Vol(D_t)}{\Vol(B^W_t)}=\sup_t\frac{\Vol(D_t)}{\Vol(B^W_t)}<\infty$. Then, since $F(t)) \geq 0\,\,\forall t>0$:
$$
\int_{t_0}^\infty F(s)ds<\infty
$$
and hence there is a monotone increasing sequence $\{t_i\}_{i=0}^\infty$ tending to infinity, such that:
\begin{equation}\label{limit}
\lim_{i\to\infty}F(t_i)=0
\end{equation}

Let us consider now  the exhaustion $\{D_{t_i}\}_{i=1}^\infty$ of $P$ by these extrinsic balls.

By using equation (\ref{CheegerConstant}), we have that,

\begin{equation}
\mathcal{I}_{\infty}(P) \leq \frac{\Vol(\partial D_{t_i})}{\Vol(D_{t_i})} \leq \frac{(\Vol(D_{t_i}))'}{\Vol(D_{t_i})}\,\,\,\,\forall r_i
\end{equation}

On the other hand, since $\lim_{i\to\infty}F(t_i)=0$, then

\begin{equation}
\lim_{i\to\infty}\frac{(\Vol(D_{t_i}))'}{\Vol(D_{t_i}))}=\lim_{i\to\infty}\frac{\Vol(S^W_{t_i})}{\Vol(B^W_{t_i})}
\end{equation}

and therefore

\begin{equation}
\mathcal{I}_{\infty}(P) \leq \lim_{i\to\infty}\frac{\Vol(S^W_{t_i})}{\Vol(B^W_{t_i})}
\end{equation}

Inequality (\ref{statcor1}) follows inmediately  taking into account that, as was shown in the examples above,  when $P$ is minimal in a Cartan-Hadamard manifold, then considering $h(r)=0 \,\,\,\,\forall r$ and considering $w(r)= w_b(r)$,  we have that $\{ N^{n}, P^{m}, \kam\}$ is a comparison constellation, with $\kam$ $w_b$-balanced from below. 

As by hypothesis,  $\Sup_{t>0}(\frac{\Vol(D_t)}{\Vol(B^{b,m}_t)})<\infty$ and we have that, 
\begin{equation}
\begin{aligned}
\lim_{t \to \infty}\frac{\Vol(S^{W}_t)}{\Vol(B^{W}_t)}&= \lim_{t \to \infty}\frac{\Vol(S^{0,m-1}_t)}{\Vol(B^{0,m}_t)} = 0 \quad \text{if $b=0$}\\
\lim_{t \to \infty}\frac{\Vol(S^{W}_t)}{\Vol(B^{W}_t)}&=\lim_{t \to \infty}\frac{\Vol(S^{b,m-1}_t)}{\Vol(B^{b,m}_t)} = (m-1)\sqrt{-b} \quad \text{if $b<0$}
\end{aligned}
\end{equation}
\noindent we now apply inequality (\ref{ineq1}).
\end{proof}

Now, we have the following result, which is a direct extension to Yau's classical result (see \cite{Y}) on minimal submanifolds, using the same techniques as in \cite{Cha3}:

\begin{theorem}\label{Main2}   
Consider an isoperimetric comparison
constellation  $\{ N^{n}, P^{m}, M^m_W
\}$. Assume that the isoperimetric comparison space
 $\, M^m_W\,$ is $w$-balanced from above.
 Assume, moreover, that  the limit $\lim_{r \to \infty}\frac{W'(r)}{W(r)} $ exists.

Then 

\begin{equation}
\mathcal{I}_{\infty}(P) \geq (m-1) \lim_{r \to \infty}\frac{W'(r)}{W(r)} \quad.
\end{equation}

In particular, let $P^m$ be a complete and minimal submanifold properly immersed in a Cartan-Hadamard manifold  $N$ with sectional curvatures bounded from above as $K_N \leq b \leq 0$.

Then 
\begin{equation}\label{statcor2}
\mathcal{I}_{\infty}(P) \geq (m-1)\sqrt{-b} \quad.
\end{equation}
\end{theorem}
\begin{proof}

From equation  (\ref{eqLambdaDiffeq}) in definition  \ref{defCspace} of the isoperimetric comparison space, we have:

\begin{equation}\label{eqLambdaDiffeq2}
\begin{aligned}
(m-1)\frac{W'(r)}{W(r)}+\eta_w(r)&= \, m\,\left(\eta_w(r) - h(r)\right)
\end{aligned}
\end{equation}

On the other hand, from Theorem  \ref{corLapComp}:
\begin{equation} \label{eqLap1}
\begin{aligned}
\Delta^{P}r \, \geq \, &\left(m-
\Vert \nabla^{P} r \Vert^{2} \right)\eta_{w}(r) +m 
\langle \, \nabla^{N}r, \, H_{P}  \, \rangle\geq\\
& (m-1)\eta_{w}(r) +m 
\langle \, \nabla^{N}r, \, H_{P}  \, \rangle\geq\\
&(m-1)\eta_{w}(r) -m\,h(r)=\\
&m\left(\eta_w(r)-h(r)\right)-\eta_{w}(r) 
\end{aligned}
\end{equation}
Then, applying  (\ref{eqLambdaDiffeq2})
\begin{equation}
\triangle^Pr\geq (m-1)\frac{W'(r)}{W(r)}
\end{equation}


Now, if we consider a domain $\Omega \subseteq P$, which is precompact and with smooth closure, we have, given its outward unitary normal vector field, $\nu$:
$$\langle \nu, \nabla^P r\rangle \leq 1$$
hence by applying divergence Theorem, and taking into account that $\frac{W'(r)}{W(r)}$ is non-increasing
\begin{equation}
\begin{aligned}
\Vol(\partial \Omega) &\geq \int_{\partial \Omega} \langle \nu, \nabla^P r \rangle d\mu \\=\int_{\Omega} \Delta^P r d\sigma \geq \int_{\Omega} \frac{W'(r)}{W(r)}d\sigma &\geq(m-1)\lim_{r \to \infty}\frac{W'(r)}{W(r)} \Vol(\Omega)
\end{aligned}
\end{equation}

As $$\frac{\Vol(\partial \Omega)}{\Vol(\Omega)} \geq (m-1)\lim_{r \to \infty}\frac{W'(r)}{W(r)}$$ for any domain $\Omega$, we have the result.

Inequality (\ref{statcor2}) follows inmediately  taking into account that, as in the proof of Theorem \ref{Main1} and in the examples above,  when $P$ is minimal in a Cartan-Hadamard manifold, then we have that $\{ N^{n}, P^{m}, \kam\}$ is a comparison constellation ($h(r)=0 \,\,\,\,\forall r$ and $w(r)= w_b(r)$), with the isoperimetric comparison space used as a model $M^m_W=\kam$ $w_b$-balanced from above. Moreover, $\lim_{r \to \infty}\frac{W'(r)}{W(r)}=\sqrt{-b}$.
\end{proof}
\bigskip

\section{Applications: Cheeger constant of minimal submanifolds of Cartan-Hadamard manifolds}

\subsection{Isoperimetric results and Chern-Osserman Inequality}
\label{secIsopRes}

This subsection provides two results which describe how minimality and the control on the total extrinsic curvature of the submanifold implies, among other topological consequences, having finite volume growth. The first (Theorem \ref{thmAnderson}) is due to M.T. Anderson, and the second (Theorem \ref{ChernOss1}) was proved in the Euclidean setting by S.S. Chern and R. Osserman, with an extension to the Hyperbolic setting due to Q. Chen. These results will be used to prove Corollaries \ref{Cor2} and \ref{Cor3} in the next Subsection \S4.2.

\begin{theorem}\label{thmAnderson}(see \cite{A1}).
Let $P^m$ be an oriented, connected and complete minimal submanifold immersed in the Euclidean space $\erre^n$. Let us suppose that $\int_P\Vert B^P \Vert^m d\sigma<\infty$, where $B^P$ is the second fundamental form of $P$. Then

\begin{enumerate}
\item $P$ has finite topological type.
\item $\Sup_{t>0}(\frac{\Vol(\partial D_t)}{\Vol(S^{0,m-1}_t)})<\infty \quad.$
\item $-\chi(P)=\int_P \Phi d\sigma+\lim_{t \to \infty}\frac{\Vol(\partial D_t)}{\Vol(S_t^{0,m-1})} \quad.$
\end{enumerate}
where $\chi(P)$ is the Euler characteristic of $P$ and $\Phi$ is the Gauss-Bonnet-Chern form on $P$, and $S_t^{b,m-1}$ denotes the geodesic $t$-sphere in $\kam$.
\end{theorem}
\begin{remark}\label{remarkone}
Note that, on applying inequality (\ref{eqIsopGeneralB}) in Theorem \ref{thmIsopGeneral1} to the submanifold $P$ the theorem above, we conclude that, under the assumptions of Theorem  \ref{thmAnderson}, we have the following bound for the volume growth
\begin{equation}
\Sup_{t>0}(\frac{\Vol(D_t)}{\Vol(B^{0,m}_t)}) \leq \Sup_{t>0}(\frac{\Vol(\partial D_t)}{\Vol(S^{0,m-1}_t)})<\infty \quad.
\end{equation}
where $B_t^{b,m}$ denotes the geodesic $t$-ball in $\kam$.
\end{remark}

On the other hand, we have that Chern-Osserman Inequality is satisfied by complete and minimal surfaces in a simply connected real space form with constant sectional curvature $b \leq 0$, $\kan$. Namely 

\begin{theorem}\label{ChernOss1}(see \cite{A1}, \cite{QCh1} and \cite{QCh3}. For an alternative proof, see \cite{GP}).  Let $P^2$ be an complete minimal surface immersed in a simply connected real space form with constant sectional curvature $b \leq 0$, $\kan$. Let us suppose that $\int_P\Vert B^P \Vert^2 d\sigma<\infty$. Then

\begin{enumerate}
\item $P$ has finite topological type.
\item $\Sup_{t>0}(\frac{\Vol(D_t)}{\Vol(B^{b,2}_t)})<\infty \quad.$
\item $-\chi(P)\leq \frac{\int_P \Vert B^P \Vert ^2}{4\pi}-\Sup_{t>0}\frac{\Vol(D_t)}{\Vol(B_t^{b,2})} \quad.$
\end{enumerate}
where $\chi(P)$ is the Euler characteristic of $P$. 
\end{theorem}

\subsection{The Corollaries}
\label{secIsopRes2}
In this subsection, we are going to state and prove the following results, which are direct consequences of the main theorems in Section \S.3 and Theorems  \ref{thmAnderson} and \ref{ChernOss1} in Subsection \S4.1.

The first Corollary \ref{Cor1} is a direct application of Theorems \ref{Main1} and \ref{Main2}.

\begin{corollary}\label{Cor1} 
 Let $P^m$ be a complete and minimal submanifold properly immersed in a Cartan-Hadamard manifold  $N$ with sectional curvatures bounded from above as $K_N \leq b \leq 0$.
 Let us suppose that $\Sup_{t>0}(\frac{\Vol(D_t)}{\Vol(B^{b,m}_t)})<\infty$

Then 
\begin{equation}
\mathcal{I}_{\infty}(P) =(m-1) \sqrt{-b} \quad.
\end{equation}
\end{corollary}
\begin{proof}
This is a direct consequence of inequalities (\ref{statcor1}) and (\ref{statcor2}) in Theorem \ref{Main1} and Theorem \ref{Main2}. 
\end{proof}

The second and the third corollaries \ref{Cor2} and \ref{Cor3} are based on Theorems \ref{thmAnderson} and \ref{ChernOss1}.

When we consider minimal submanifolds in $\erre^n$, we have the following result:
\begin{corollary}\label{Cor2} 
 Let $P^m$ be a complete and minimal submanifold properly immersed in $\erre^n$, with finite total extrinsic curvature $\int_P\Vert B^P \Vert^m d\sigma<\infty$.

Then 
\begin{equation}
\mathcal{I}_{\infty}(P) =0 \quad.
\end{equation}
\end{corollary}
\begin{proof}
In this case,  taking  $h(r)=0 \,\,\,\forall r$ and $w_0(r)= r$, we have that $\{ \erre^{n}, P^{m}, \erre^{m}\}$ is a comparison constellation bounded from above, with $\erre^m$ $w_0$-balanced from below. Hence, we apply Theorem \ref{thmIsopGeneral1} to obtain
\begin{equation}\label{isopBound}
\frac{\Vol(D_t)}{\Vol(B^{0,m}_t)} \leq
\frac{\Vol(\partial D_t)}{\Vol(S^{0,m-1}_t)} \,\,\,\,\,\textrm{for all}\,\,\,
t>0  \quad .
\end{equation}

Therefore, as the total extrinsic curvature of $P$ is finite, by applying Theorem \ref{thmAnderson}, inequality (\ref{isopBound}) and Remark \ref{remarkone}, we have
$$\Sup_{t>0}(\frac{\Vol(D_t)}{\Vol(B^{0,m}_t)})<\infty$$
Finally,
$$\lim_{t \to \infty}\frac{\Vol(S^{0,m-1}_t)}{\Vol(B^{0,m}_t)} = \lim_{t \to \infty} \frac{m}{t} =0$$

Hence, applying Theorem \ref{Main1},  $\mathcal{I}_{\infty}(P) \leq 0$, so $\mathcal{I}_{\infty}(P)= 0$.
\end{proof}

 Corollary \ref{Cor2} can be extended to complete and minimal surfaces (properly) immersed in the  Hyperbolic space, with finite total extrinsic curvature: 
 
\begin{corollary}\label{Cor3} 
 Let $P^2$ be a complete and minimal surface immersed in $\kan$ with finite total extrinsic curvature $\int_P\Vert B^P \Vert^2 d\sigma<\infty$.

Then 
\begin{equation}
\mathcal{I}_{\infty}(P) =\sqrt{-b} \quad.
\end{equation}
\end{corollary}
\begin{proof}
As the total extrinsic curvature of $P$ is finite, by applying Theorem \ref{ChernOss1} we have:
$$\Sup_{t>0}(\frac{\Vol(D_t)}{\Vol(B^{b,2}_t)})<\infty$$

Then, apply Corollary \ref{Cor1} with $m=2$.
\end{proof}

\end{document}